\documentclass{amsart}

\numberwithin{equation}{section}

\usepackage{longtable}

\usepackage{color}

\usepackage{booktabs}
\usepackage{hyperref}
\hypersetup{
    colorlinks,%
    citecolor=blue,%
    filecolor=blue,%
    linkcolor=blue,%
    urlcolor=blue
}
\usepackage[all]{xy}
\usepackage[alphabetic,backrefs,lite]{amsrefs}
\usepackage{ amssymb,mathrsfs }

\newtheorem{introthm}{Theorem}

\newtheorem{theorem}{Theorem}[section]
\newtheorem{lemma}[theorem]{Lemma}
\newtheorem{proposition}[theorem]{Proposition}
\newtheorem{corollary}[theorem]{Corollary}
\theoremstyle{definition}

\newtheorem{example}[theorem]{Example}

\newtheorem{remark}[theorem]{Remark}
\theoremstyle{remark}

\newcommand{\ff}{{\mathbb F}}
\newcommand{\zz}{{\mathbb Z}}
\newcommand{\qq}{{\mathbb Q}}
\newcommand{\pp}{{\mathbb P}}

\newcommand{\GL}{\operatorname{GL}}
\newcommand{\Gal}{\operatorname{Gal}}
\newcommand{\Pic}{\operatorname{Pic}}
\newcommand{\Aut}{\operatorname{Aut}}

\newcommand{\NS}{\operatorname{NS}}

\begin{document}
\title{On B\"uchi's K3 surface}
\author{Michela Artebani}
\address{
Departamento de Matem\'atica, 
Universidad de Concepci\'on, 
Casilla 160-C,
Concepci\'on, Chile}
\email{michela.artebani@gmail.com, antonio.laface@gmail.com}

\author{Antonio Laface}

\author{Damiano Testa}
\address{Mathematics Institute, Zeeman Building, University of Warwick, 
Coventry CV4 7AL, United Kingdom}
\email{adomani@gmail.com}


\thanks{The first author has been partially 
supported by Proyecto FONDECYT Regular N. 1110249.
The second author has been partially supported 
by Proyecto FONDECYT Regular N. 1110096.
The third author has been partially supported by EPSRC grant EP/F060661/1.\\
}

\begin{abstract}
We study the B\"uchi
K3 surface proving that it belongs to the one dimensional 
family of Kummer surfaces associated to genus two curves 
with automorphism group $D_4$.  
We compute its Picard lattice and show that 
the rational points of the surface are Zariski-dense.
Moreover, we provide analogous results for the Kummer 
surface associated to any genus two curve whose automorphism 
group contains a non-hyperelliptic involution.
\end{abstract}
\maketitle

\section*{Introduction}
Let $n$ be an integer with $n \geq 3$ and let $\{x_i\}_{i=1}^n$ be a 
sequence of $n$ integers satisfying the system of second 
order difference equations 
\begin{equation}\label{diff}
 (x_{i+2}^2-x_{i+1}^2)-(x_{i+1}^2-x_i^2) = 2
 \qquad\text{ for }
 i\in\{1,\dots,n-2\}.
\end{equation}
Any sequence obtained from a solution of~\eqref{diff} 
by arbitrary sign changes or reversal is still a solution.  
Also, for any integer $x$, the sequence of consecutive 
integers $\{x_i=x + i \}_{i=1}^n$ is a solution of~\eqref{diff}.
An integral solution of~\eqref{diff} is {\em trivial} if it is 
obtained from a sequence of consecutive integers by 
arbitrary sign changes or reversal, and it is 
{\em non-trivial} otherwise.

The B\"uchi problem is the well known question asking 
whether there exists a positive integer $n$ such that all 
integral solutions of~\eqref{diff} are trivial.  It is known 
that if $n\leq 4$ there are non-trivial integral solutions
of~\eqref{diff}.

A positive answer to B\"uchi's problem 
would imply, using the negative answer 
to Hilbert's Tenth Problem by Yu.\ Matiyasevich, 
that there is no algorithm to decide 
whether a system of diagonal quadratic 
forms with integer coefficients 
admits an integer solution.
Several authors studied 
B\"uchi's problem and related questions.
In  \cite{Vo}  P.\ Vojta determined all 
the one-parameter solutions to \eqref{diff} 
for $n\geq 8$; combining this result with the 
Bombieri-Lang conjecture 
gives a positive answer to B\"uchi's problem.
An explicit description for length three and four
B\"uchi sequences is provided in \cite{Vi} and \cite{SV}.
The problem has been studied also 
on finite fields, see for example the 
paper by H.\ Pasten~\cite{Pa}. 
In~\cite{BB} the authors 
study the related problem of
finding sequences of integer squares
with constant second differences.
They find infinitely many non-trivial 
such sequences 
lying on elliptic curves of positive rank.
For a recent survey on the problem, we refer 
the reader to~\cite{PPV}.

For any $n\geq 3$ the system~\eqref{diff} 
defines a projective surface $X_n$ of degree $2^{n-2}$ 
in $\pp^n$. This is a rational surface for $n=3$ and $4$,
a K3 surface for $n=5$ and a surface of general type 
otherwise.
In this paper we show that the B\"uchi K3
surface $X_5$ is the minimal resolution of the 
Kummer surface of a genus 
two curve whose automorphism group 
is isomorphic to $D_4$. 
Such curves form a one-dimensional 
family in $\mathcal{M}_2$, so that $X_5$ naturally 
lives in a family of lattice polarized 
K3 surfaces. 
We determine the Picard lattice of a general member
of this family and show that it coincides 
with $\Pic(X_5)$. Moreover, we perform
a similar analysis for any curve in
$\mathcal{M}_2$ whose automorphism
group contains a non-hyperelliptic involution.
There are five irreducible families  of such curves,
determined by their automorphism group ~\cite{CGLR}.

\begin{introthm}\label{Pic}
Let $G$ be one of the groups in Table~\ref{tab:aut}, let 
$\Lambda_G$ be the corresponding lattice and
let $\mathscr{F}_G$ be the family of 
genus two curves whose automorphism
group is isomorphic to $G$.
If $C\in\mathscr{F}_G$ and $X$ is 
the minimal resolution of 
the Kummer surface of $C$, 
then $\Pic(X)$ contains a primitive 
sublattice isometric to $\Lambda_G$.
Moreover, $\Pic(X)$ is isomorphic
to $\Lambda_G$ whenever the 
two lattices have the same rank; this happens 
for very general curves $C$ in the family $\mathscr{F}_G$.
\end{introthm}

\begin{table}[h]
\begin{tabular}{l|l}
$G$  & $\Lambda_G$\\
\midrule
$V_4$  & $U(2)\oplus E_8\oplus D_7\oplus (-4)$\\[2pt]
$D_4$  & $U(4)\oplus E_8^2\oplus (-4)$\\[2pt]
$D_6$  & $U(2)\oplus E_8^2\oplus (-12)$\\[2pt]
$2D_6$ & $U\oplus E_8^2\oplus (-4)\oplus (-12)$\\[2pt]
$\tilde{S}_4$ & $U\oplus E_8^2\oplus (-4)\oplus (-8)$.
\end{tabular}
\caption{Automorphism group and Picard lattices of Kummer surfaces} \label{tab:aut}
\end{table}

As a consequence of this result we 
determine the Picard lattice of $X_5$
and of other K3 surfaces studied 
from an arithmetical point of view
(see Examples~\ref{bre} and~\ref{bb}).
We prove that the
set of non-trivial {\em{rational}}
points of the B\"uchi K3 surface
$X_5$ is infinite and Zariski-dense.
The strategy is to find an elliptic fibration 
defined over $\qq$ with infinitely many sections 
defined over $\qq$.
In fact, we prove that the same statement holds 
for all surfaces in any of the above families
(see Theorem~\ref{dense}).

The paper is organized as follows.
In Section~\ref{sebu} we introduce the B\"uchi
surfaces~$X_n$, characterise the fields over which
they are smooth and show that, for $n\geq 4$,
the linear automorphism group of $X_n$ consists of scalar permutations.
Section~\ref{se-g2} deals with genus two
curves admitting a non-hyperelliptic
involution. Here we determine the
N\'eron-Severi lattice of their Jacobians.
In Section~\ref{se-ks} we prove Theorem~\ref{Pic}
and Theorem~\ref{dense}.
In Section~\ref{seku} we provide a detailed
study of the Mordell-Weil groups of
some elliptic fibrations $X_5\to\pp^1$
defined over the rational numbers.
We also analyse 
the Galois representation of 
the elliptic curve $E$ and its set of 
supersingular primes: we use these results to 
compute the local zeta function of $X$ at such primes.

\subsection*{Acknowledgements}
We thank the anonymous referee 
for his careful reading of our paper and for
suggesting to us to look at families of lattice
polarized K3 surfaces.
We would also like to thank Kieran O'Grady 
and Daniel Huybrechts for their suggestions 
and their help with the Mukai correspondence 
and moduli spaces of sheaves on K3 surfaces,
as well as Matthias Sch\"utt, Pol Vanhaecke 
and Xavier Vidaux for comments on a preliminary version of this paper.

\section{B\"uchi surfaces} \label{sebu}
Let $n$ be an integer with $n \geq 3$.  The {\em B\"uchi surface} $X_n$ is the 
zero locus in $\pp^{n}$ of the
$n-2$ quadrics of equations
\begin{equation}\label{eq-buchi}
 x_{i+2}^2-2\,x_{i+1}^2+x_i^2-2\,x_0^2
 = 0
 \qquad
 \text{for }i\in\{1,\dots,n-2\}.
\end{equation}
Observe that $X_3$ is a quadric,
$X_4$ is a del Pezzo surface of degree 
four, $X_5$ is a K3 surface 
and $X_n$, $n\geq 6$, is a surface of general type (if they are smooth).
A first attempt to find rational
points on $X_n$ is to parametrize the 
surface when $n\leq 4$.
The del Pezzo surface $X_4$  
can be parametrized by the 
system of plane cubics passing 
through the five points
\[
 [1, 0, 0]
 \quad
 [0 , 1 , 0]
 \quad
 [0 , 0 , 1]
 \quad 
 [3 , -3 , 2]
 \quad
 [3 , -1 , 1].
\]
Explicitly we have the following 
parametrization of $X_4$:
\[
 \begin{array}{r@{\;\; ~=~ \;\;}l}
  x_0 & -2a^2b - 6ab^2 - 3a^2c + 9b^2c + 9ac^2 + 9bc^2\\[5pt]
  x_1 & -12ab^2 + 3a^2c - 18abc + 9b^2c - 18ac^2\\[5pt]
  x_2 & 2a^2b - 6ab^2 - 24abc - 9ac^2 + 9bc^2\\[5pt]
  x_3 & -4a^2b - 3a^2c + 6abc - 9b^2c - 18bc^2\\[5pt]
  x_4 & -6a^2b -6 ab^2 -6 a^2c - 18b^2c + 9ac^2 + 27bc^2
  \end{array}
\]
which immediately proves
the existence of infinitely
many non-trivial rational
solutions of the B\"uchi's 
problem for four variables. 
Indeed, not only the set of rational 
points of $X_4$ is infinite,
but even the set of points that are integral 
with respect to the divisor $x_0=0$ is infinite (see~\cite{PPV}).
The surface $X_4$ contains exactly 
 $16$ lines whose parametric equations in the affine chart 
$x_0\not=0$ are given by 
\[
x_i=\pm(t+i),\quad i \in \{1,2,3,4\}.
\]

The surface $X_5$ is the
double cover of $X_4$ 
branched along the smooth genus $5$ curve
$B\in |{-2K_{X_4}}|$ defined in $X_4$ by 
$x_3^2-2x_4^2=2x_0^2$ (the covering map 
is the projection on the first four coordinates).
Since the lines of $X_4$ are tangent
to the branch locus $B$,
the inverse image of each line 
in the double cover $X_5$ decomposes as the 
union of two lines.
Thus $X_5$ contains $32$ lines, 
with parametric equations 
\begin{equation}\label{lines}
x_i=\pm(t+i),\quad i \in \{1,2,3,4,5\}.
\end{equation}
These are the only lines in $X_5$, 
since the covering map
is linear. Observe that the lines 
on $X_5$ contain all
the trivial solutions 
of B\"uchi's problem.

\subsection{Smoothness of B\"uchi surfaces}

We show that all B\"uchi surfaces
are smooth in characteristic zero and provide
a necessary and sufficient condition for their smoothness
in positive characteristic. 
For any positive integer $n$ we define 
a pair of matrices
\[
B_n =  \begin{pmatrix}
-2 & 1 & -2 & 1 \\
-2 & & 1 & -2 & 1 \\
\vdots & & & \ddots & \ddots & \ddots \\
-2 & & & & 1 & -2 & 1
\end{pmatrix} 
\qquad
A_n = \begin{pmatrix}
-2 & 1 \\
1 & -2 & 1 \\
& \ddots & \ddots & \ddots \\
& & 1 & -2 & 1 \\
& & & 1 & -2 
\end{pmatrix}
\]
where $B_n$ is a $(n-2)\times (n+1)$
matrix while $-A_n$ is the Cartan matrix
of size $n$ of a root system of type $A$.
We recall that the determinant of the matrix $A_n$ 
equals $(-1)^n (n+1)$: this is well-known and it is also immediate to prove by induction developing the determinant with respect to the first row.  In particular, if $k$ is a field, then the rank of the matrix $A_n$ satisfies 
\[
{\rm rk} (A_n) = \left\{
\begin{array}{ll}
n & {\textrm{if ${\rm char}(k)$ does not divide $n+1$,}} \\[5pt]
n-1 & {\textrm{otherwise.}}
\end{array}
\right.
\]
It is easy to check that the following 
identity holds
\begin{equation} \label{relan}
\bigl( 1 \;\; 2 \;\; 3 \;\; \cdots \;\; n \bigr) A_n = \bigl( 0 \;\; 0 \;\; \cdots \;\; 0 \;\; -(n+1) \bigr).
\end{equation}

\begin{lemma} \label{detjac}
Let $a,b,c$ be integers satisfying $0 \leq a < b < c \leq n$ and let $B_n^{a,b,c}$ denote the $(n-2)\times (n-2)$-matrix obtained from $B_n$ by removing the three columns with indices $a+1,b+1,c+1$.  
If the index $a$ satisfies $a \geq 1$, then the determinant of the matrix $B_n^{a,b,c}$ equals 
\[
\det(B_n^{a,b,c}) = (-1)^{a+b+c}(a-b)(a-c)(b-c) .
\]
If the index $a$ satisfies $a = 0$, then the determinant of the matrix $B_n^{b,c} = B_n^{0,b,c}$ equals 
\[
\det(B_n^{b,c}) = (-1)^{b+c}(b-c) .
\]
\end{lemma}

\begin{proof}
We only treat the case in which $a$ is different from $0$, since the remaining case is similar and simpler.  Fix integers $a,b,c$ satisfying $1 \leq a < b < c \leq n$ and proceed by induction on $n$ to prove the assertion.  If $a \geq 2$, then we develop the determinant of $B_n^{a,b,c}$ with respect to the second column to find $\det(B_n^{a,b,c}) = -\det(B_{n-1}^{a-1,b-1,c-1})$.  Similarly, if $c \leq n-1$, then we develop the determinant of $B_n^{a,b,c}$ with respect to the last column to find $\det(B_n^{a,b,c}) = \det(B_{n-1}^{a,b,c})$.  Therefore, by the inductive hypothesis we reduce to the case in which the indices satisfy $a=1$, $c=n$ and the matrix $B_n^b = B_n^{1,b,n}$ is the matrix 
\[
B_n^b = 
\begin{pmatrix}
\vdots & & & \framebox{$A_{b-2}$} \\
-2 & 0 & \cdots & 0 \;\;\; 1 & 1 \;\;\; 0 \; \cdots \\
\vdots & & & & \framebox{$A_{n-b-1}$} 
\end{pmatrix} ,
\]
where the boxes denote square blocks of sizes $b-2$ and $n-b-1$ filled by the matrices $A_{b-2}$ and $A_{n-b-1}$ respectively.  If, for $i \in \{1,\ldots,b-1\}$, we multiply the $i$-th row of the matrix $B_n^b$ and then sum the resulting row vectors, then, using relation~\eqref{relan}, we obtain the row vector $v_1 = -(b-1) \bigl( (b-2) \;\; 0 \;\; 0 \;\; \cdots \;\; 0 \;\; 1 \;\; 0 \;\; \cdots 0 \bigr)$.  A similar computation shows that row vector $v_2 = -(n-b) \bigl( (n-b-1) \;\; 0 \;\; 0 \;\; \cdots \;\; 0 \;\; 1 \;\; 0 \;\; \cdots 0 \bigr)$ is a linear combination of the last $n-b-1$ rows.  Hence, adding $\frac{1}{b-1}v_1 + \frac{1}{n-b}v_2$ to the $(b-1)$-st row of the matrix $B_n^b$, the determinant is not affected and the matrix $B_n^b$ becomes the matrix 
\[
\begin{pmatrix}
\vdots & & & \framebox{$A_{b-2}$} \\
-(n-1) & 0 & \cdots & 0 \;\;\; 0 & 0 \;\;\; 0 \; \cdots \\
\vdots & & & & \framebox{$A_{n-b-1}$}
\end{pmatrix} ,
\]
whose determinant is $(-1)^{1+b+n}(1-n)(1-b)(b-n)$, as required.
\end{proof}

Our next result is about the smoothness of the scheme $X_n$.  To simplify the notation, let $Y_n$ denote the linear subspace of $\mathbb{P}^n$ defined by the system of equations 
\[
Y_n \colon \left\{
\begin{array}{rcl}
x_i-2x_{i+1}+x_{i+2} & = & 2x_n , \\[5pt]
{\textrm{for }}i & \in & \{1,\ldots,n-2\}.
\end{array}
\right.
\]
The dimension of the linear subspace $Y_n$ is two and there is a morphism $q \colon X_n \to Y_n$ determined by $q([x_0,\ldots,x_n]) = [x_0^2,\ldots,x_n^2]$.  The morphism $q$ is finite of degree $2^{n}$.

\begin{corollary}
The surface $X_n$ is singular if and only if 
the characteristic $p$ of the ground field satisfies 
$0<p<n$.  If the surface $X_n$ is singular, then $X_n$ has exactly $2^{n-p}$ irreducible components each isomorphic to $X_p$.
\end{corollary}

\begin{proof}
An easy induction shows that, for $i \in \{1,\ldots,n\}$, the equations 
\[
x_i^2 = (i-1)(i-2)x_0^2 - (i-2)x_1^2+(i-1)x_2^2
\]
hold on $X_n$.  In particular, if $i,j \in \{1,\ldots,n\}$ satisfy $i \equiv j \pmod{p}$, then the equation $x_i^2 = x_j^2$ holds on $X_n$.  We deduce that to prove the corollary, it suffices to show that $X_n$ is smooth under the assumption $p \notin \{2,3,\ldots,{n-1}\}$.

Let $x=[x_0,\ldots,x_n]$ be a point in $X_n$.  To check smoothness, we compute the rank of the Jacobian matrix $J_n$ of the equations~\eqref{eq-buchi} at the point $x$.  Observe that the rank of the matrix $J_n$ at $x$ depends only on the positions of the vanishing coordinates of the point $x$.
As we observed above, we reduce to the case in which $p$ satisfies $p \notin \{2,\ldots,{n-1}\}$.

If at most two of the coordinates of $x$ are non-zero, then let $a,b \in \{0,1,\ldots,n\}$ be indices such that $a<b$ and for all $i \in \{0,\ldots,n\} \setminus \{a,b\}$ we have $x_i \neq 0$.  Let $c$ be an integer satisfying $c \in (\{1,\ldots,n\} \cap \{b-1,b+1\}) \setminus \{a\}$.  Lemma~\ref{detjac} implies that the matrix obtained from $B_n$ by removing the columns with indices $a,b,c$ has non-vanishing determinant and the result follows in this case.

To conclude, it therefore suffices to show that, assuming $p \notin \{2,\ldots,{n-1}\}$, there are no points on $X_n$ with three vanishing coordinates.  Suppose that the integers $a,b,c \in \{0,\ldots,n\}$ satisfy $0 \leq a < b < c \leq n$ and also $x_a,x_b,x_c = 0$.  The determinant of the matrix $B_n$ with the columns $a,b,c$ removed is non-zero by Lemma~\ref{detjac}, since there are no indices in $\{1,\ldots,n\} \subset \{1,2,\ldots,p\}$ with difference divisible by $p$.  It follows that the subscheme of $Y_n$ defined by $x_a=x_b=x_c=0$ is empty and the corollary follows.
\end{proof}

\subsection{Linear automorphisms}\label{auto}
A \emph{scalar permutation} is an automorphism of the projective space $\mathbb{P}^N$ whose matrix with respect to the natural homogeneous coordinates only has one non-zero entry in each row and in each column.  Equivalently, a scalar permutation is an automorphism of projective space fixing the set of coordinate points (namely, those points of projective space with a unique non-zero coordinate).

\begin{theorem}\label{aut}
Let $X\subset\pp^N$ be a non-degenerate 
smooth projective variety of codimension at least two
whose defining ideal is a complete intersection 
of diagonal hypersurfaces of the same degree $d\geq 2$.
Then the linear automorphism 
group of $X$ consists of scalar permutations.
\end{theorem}
\begin{proof}
If $\mathcal{Q} \subset \mathbb{P} \bigl( I(X)_d \bigr)$ 
is a linear space of hypersurfaces of degree $d$ vanishing on $X$, 
denote by $\Sigma_\mathcal{Q}$ the intersection 
of the vertices of the cones   
in $\mathcal{Q}$, and note that $\Sigma_\mathcal{Q}$ 
is a coordinate linear subspace. 
Therefore, the linear automorphism group of $X$ 
permutes the set of linear subspaces 
$\bigl\{ \Sigma_\mathcal{Q} ~\mid~ \mathcal{Q} \subset \mathbb{P} \bigl( I(X)_d \bigr) \bigr\}$.
Hence, to prove the assertion, it suffices to show that 
all coordinate points of $\mathbb{P}^N$ 
appear among the linear subspaces $\Sigma_\mathcal{Q}$.
To prove this it is enough to show that for any
subset $S$ of $\dim X+2$ variables
of the homogeneous coordinate ring of $\pp^N$
there is a diagonal element $f\in I(X)_d$
containing exactly the variables in $S$:
given such an $S$ we eliminate the complementary
variables from $I(X)_d$ obtaining a homogeneous
form $f\in I(X)_d$. 
The vertex of the cone of $V(f)$ is a linear
subspace $L$ of dimension $N-r$. Since
$X$ is a smooth complete intersection 
it follows that $X\cap L=\emptyset$ and thus
$r\geq \dim X+2$. We conclude that $r=\dim X+2$, as required.
\end{proof}

As an application of Theorem~\ref{aut}, we deduce the following corollary.

\begin{corollary}
Let $X_n$ be a smooth B\"uchi surface with
$n\geq 4$. Then the linear automorphism group
of $X_n$ consists of scalar permutations.
\end{corollary}

\begin{remark}
Each surface $X_n$ is a complete
intersection of quadrics. 
Using the adjunction 
formula, we see that the canonical divisor 
$X_n$ is linearly equivalent 
to $(n-5)H$, where $H$ is the hyperplane 
section of $X_n$.  Every automorphism 
of $X_n$ must preserve the (anti)canonical 
linear series.   
Since the Picard group of $X_n$ is 
torsion-free, it follows that the automorphism 
group of $X_n$ consists just of the linear 
automorphisms of $X_n$ when $n \geq 3$ and $n \neq 5$.

If $n=5$ the previous argument
is no longer true, since 
$K_{X_5}$ is trivial. In this
case the automorphism group
of $X_5$ has infinite order
and we will show in Section~\ref{sub:ratpt}
that this is true even for the subgroup
$\Aut(X_5)_\qq$ of automorphisms
defined over the rationals.
\end{remark}

Let $\rho \colon \mathbb{P}^n \to \mathbb{P}^n$ be the involution defined by 
\[
\rho([x_0,\ldots,x_n]) = [x_0,x_n,x_{n-1},\ldots,x_2,x_1];
\]
the automorphism $\rho$ induces an involution of $X_n$ that we denote by the same symbol.

For $i \in \{0,\ldots,5\}$, denote by $\sigma_i \colon X_5 \to X_5$ the automorphism induced by the sign change of the variable $x_i$ and by $\tau$ the involution
\[
 \tau\colon X_5\to X_5
 \qquad
 [x_0,\dots,x_4]\mapsto
 [x_3, 2x_2 , x_1 , 2x_0 , x_5 , 2x_4].
\]
As a consequence of Proposition \ref{autofam} 
(or also by a direct computer analysis),
we find that 
the linear automorphism group of $X_5$ is generated by 
$\rho,\sigma_0,\ldots,\sigma_5$ and $\tau$;
these elements are non-symplectic involutions.  
Thus, any product of an even number of these generators is symplectic.  
The product of three distinct elements among $\sigma_0 , \ldots , \sigma_5$ 
is a non-symplectic involution with no fixed points; 
the corresponding quotient of $X_5$ is an Enriques surface.  
Finally, the fixed locus of $\rho$ is a smooth genus one curve
and the fixed locus of $\tau$ is
the union of two disjoint conics.

\section{Genus two curves} \label{se-g2}
Let $C$ be a curve of genus two and let 
$\sigma$ be the hyperelliptic involution on $C$.
By~\cite{Ig} (see also~\cite{CGLR}*{Table~1})
the automorphism
group of $C$ is isomorphic to 
one of the following groups:
$\zz/2\zz$,
$(\zz/2\zz)^2$,
$D_4$,
$D_{6}$,
$2D_{6}$,
$\tilde{S}_4$,
$\zz/10\zz$.
In all but the first and last case, the group $\Aut(C)$
contains a non-hyperelliptic
involution: we fix one and denote it by~$\tau$. 
We denote by $E_1$ and $E_2$ the
two quotients
\begin{equation}
\label{E12}
 E_1 = C/\langle\tau\rangle
 \qquad \textrm{and} \qquad
 E_2=C/\langle\tau\sigma\rangle,
\end{equation}
which are both curves of genus one,
and by  $\pi_i \colon C\to E_i$
the natural quotient maps.
The main result of this section 
is the following theorem.

\begin{theorem} \label{NS}
Let $G$ be one of the groups in Table~\ref{tab:fam}, let $L_G$ be the 
corresponding lattice and let 
$\mathscr{F}_G$ be the family of 
genus two curves whose automorphism
group is isomorphic to $G$.
If $C\in\mathscr{F}_G$, then the N\'eron-Severi 
lattice of its Jacobian surface 
contains a primitive sublattice 
isometric to the lattice $L_G$.  
Moreover, the lattices $\NS(JC)$ and $L_G$ are 
isomorphic whenever they 
have the same rank; this happens 
for very general curves
$C$ in the family $\mathscr{F}_G$.
\end{theorem}

\begin{table}[h]
\begin{tabular}{l|l}
$G$  & $L_{G}$\\
\midrule
$V_4$  & $(2)\oplus (-2)$\\[2pt]
$D_4$ &  $U(2)\oplus (-2)$\\[2pt]
$D_6$ & $U\oplus(-6)$\\[2pt]
$2D_6$ & $U\oplus (-2)\oplus (-6)$\\[2pt]
$\tilde{S}_4$ & $U\oplus (-2)\oplus (-4)$
\end{tabular}
\caption{Automorphism groups and Picard lattices of Jacobian surfaces} \label{tab:fam}
\end{table}

\begin{lemma}
\label{ell}
Let $C$ be a smooth genus two curve
with a non-hyperelliptic 
involution $\gamma\in\Aut(C)$.
The subset $E_\gamma
=\{[p+\gamma(p) - K_C]: p\in C\}$ of $JC$
is an elliptic curve isomorphic to
$C/\langle\gamma\rangle$
which passes through the origin.
\end{lemma}
\begin{proof}
The curve $E_\gamma$ is the image of the 
morphism $C\to JC$ defined by 
$p\mapsto [p+\gamma(p)-K_C]$. Since 
this morphism is a double cover which 
factors through $C/\langle\gamma\rangle$
and $\gamma$ is not the hyperelliptic
involution, the curve $E_\gamma\cong C/\langle\gamma\rangle$ 
is elliptic.
The involution $\sigma\gamma$ has
two fixed points exchanged by $\gamma$,
let us call them $p,\gamma(p)$. 
The identity $[p+\gamma(p)-K_C]
=[p+\sigma(p)-K_C]=0$ shows 
that $E_\gamma$ passes through
the origin.
\end{proof}

Observe that by Lemma~\ref{ell} the curves
$E_1$ and $E_2$ of~\eqref{E12} are isomorphic
to $E_\tau$ and $E_{\sigma\tau}$ respectively.
We introduce the isogeny:
\begin{equation}
\label{isog}
 \xymatrix{
  E_\tau\times E_{\sigma\tau}\ar[r]^-\alpha & JC
  \qquad
  (x,y)\mapsto x+y.
 }
\end{equation}

\begin{lemma}\label{iso}
With the above notation, the following hold:
\begin{enumerate}
\item \label{quapro}
$E_\tau^2=E_{\sigma\tau}^2=0$ and $E_\tau\cdot E_{\sigma\tau}=4$;
\item \label{isoker}
the isogeny $\alpha$ has degree $4$ and kernel isomorphic to $V_4$;
\item \label{theequ}
the divisor $E_{\tau}+E_{\sigma\tau}$ is linearly equivalent
to twice the theta divisor on $JC$.
\end{enumerate}
\end{lemma}
\begin{proof}
Since $E_\tau$ and $E_{\sigma\tau}$ are elliptic curves in $JC$,
they both have self-intersection zero.
Moreover $E_\tau$ and $E_{\sigma\tau}$ intersect
transversally, since two elliptic curves meeting
non transversally in an abelian variety
coincide; in particular the intersection number $E_\tau\cdot E_{\sigma\tau}$ equals
the cardinality of $E_\tau\cap E_{\sigma\tau}$. This set
consists of the origin plus the set 
of solutions of the equation
$p+\tau(p)=q+\sigma\tau(q)$ for $p,q\in C$.
The last set contains the elements 
of the form $p+\tau(p)$, where $p\in {\rm Fix}(\sigma)$.
These give exactly three points 
since $\tau$ acts with three orbits of cardinality two
on the set  of Weierstrass points of $C$. 
Thus $E_\tau\cdot E_{\sigma\tau}=4$,
proving~\eqref{quapro}.

To show~\eqref{isoker} observe that 
if $(x,y)$ is in $\ker(\alpha)$, then 
$[p+\tau(p)-K_C]=x=-y=[\sigma(q)+\sigma^2\tau(q)-K_C]
=[q'+\sigma\tau(q')-K_C]$.
Thus the equation has four
solutions corresponding to 
the four points in $E_\tau\cap E_{\sigma\tau}$.
This proves~\eqref{isoker}.

Now we show~\eqref{theequ}. 
First of all, we calculate the intersection
of a theta divisor $\Theta$ with $E_\tau$ and
$E_{\sigma\tau}$. Let $q\in C$ be a fixed
point for the automorphism $\tau\in\Aut(C)$:
We denote by $C_q:=\{[p-q] : p\in C\}
=\{[p+\sigma(q)-K_C] : p\in C\}$
and observe that $C_q$ is numerically equivalent
to $\Theta$. By a similar argument as above
the set $E_\tau\cap C_q = \{[p+\tau(p)-K_C] : p\in{\rm Fix}(\tau)\}$ 
has cardinality two. On the one hand  
$\tau^*|_{E_\tau}={\rm id}$
and thus $\tau^*$ acts as the identity
on the tangent space of $E_\tau$ at
all of its points. On the other hand
$\tau^*|_{C_q}=\tau$ acts as $-{\rm id}$ on the
tangent space of $C_q$ at each 
fixed point of $\tau$. Thus we conclude
that the intersection $E_\tau\cap C_q$
is transverse so that $E_\tau\cdot C_q=2$ and
similarly one shows $E_{\sigma\tau}\cdot C_q=2$.
We deduce that the images of $E_\tau$ and
$E_{\sigma\tau}$ under the Kummer map
$\varphi_{|2\Theta|}\colon JC\to\pp^3$
are two conics.
By~\eqref{quapro} the two conics intersect
at four points since the four intersection
points of $E_\tau\cap E_{\sigma\tau}$
are fixed by $\sigma$. Thus the two conics
are coplanar and hence $2\Theta\sim
E_\tau+E_{\sigma\tau}$.

\end{proof}

\begin{proof}[Proof of Theorem~\ref{NS}]
For each group $G$ in Table~\ref{tab:fam},
we compute the  N\'eron-Severi group of 
a very general curve $C$ in $\mathcal F_G$
(see Remark~\ref{vg}).

\subsubsection*{Case~1: $G= V_4$}
In this case, $\Aut(C) = \{ id , \tau , \sigma , \tau\sigma \}$,
where the hyperelliptic involution is
$\sigma$.
By Lemma~\ref{iso} the N\'eron-Severi lattice $\NS(JC)$ 
contains the sublattice $\Lambda_0$ generated 
by the classes of $\frac{1}{2}(E_\tau-E_{\sigma\tau})$
and $\frac{1}{2}(E_\tau+E_{\sigma\tau})\equiv\Theta$.
Since $E_\tau\cdot E_{\sigma\tau}=4$, the lattice $\Lambda_0$ is 
isometric to the lattice $A_1\oplus A_1(-1)$.
The discriminant group of 
$\Lambda_0$ is isomorphic to 
$\zz/2\zz\oplus \zz/2\zz$ with quadratic form 
${\rm diag}(-\frac{1}{2},\frac{1}{2})$, therefore 
the only proper overlattice of $\Lambda_0$ 
of rank two is obtained adding 
the class of $\frac{1}{2}E_{\sigma\tau}$.
This possibility is ruled out 
by~\cite{Ka}*{Theorem~1.1},
so that $\Lambda_0$ is a primitive
sublattice of $\NS(JC)$ of rank two.

\subsubsection*{Case~2: $G= D_4$}
In this case, $\Aut(C)=\langle\tau,\eta\mid \eta^4,\tau^2,(\tau\eta)^2\rangle$,
where the hyperelliptic involution is
$\sigma=\eta^2$.
Observe that the automorphism $\eta$ 
induces an isomorphism 
between $E_{\tau}$ and $E_{\sigma\tau}$.
According to Lemma~\ref{ell}
the involution $\eta\tau\in\Aut(C)$ defines
the elliptic curve $E_{\eta\tau}$ in $JC$.
The intersection $E_\tau\cap E_{\eta\tau}$
consists of the origin plus the points $[p+\tau(p)-K_C]$,
where $p\in {\rm Fix}(\eta)$. The fixed points
of $\eta$ are contained in the Weierstrass 
points ${\rm Fix}(\sigma)$ of $C$.
By applying the Riemann-Hurwitz formula
to the degree four cyclic covering
$C\to C/\langle\eta\rangle\cong\pp^1$
we deduce that $\eta$ has exactly two
fixed points exchanged by $\tau$.
Thus we deduce $E_\tau\cdot E_{\eta\tau}=2$
and similarly one shows 
$E_{\sigma\tau}\cdot E_{\eta\tau}=2$.
The sublattice $\Lambda_1\subseteq\NS(JC)$ 
generated by the classes of $E_\tau$, 
$E_{\eta\tau}$ and $\Theta-E_\tau-E_{\eta\tau}$
 is isometric to $U(2)\oplus (-2)$.
The only isotropic and non-zero vectors 
in the discriminant group of 
$\Lambda_1$ are $\frac{1}{2} E_\tau$
and $\frac{1}{2} E_{\eta\tau}$.
These do not belong to $\NS(JC)$ by~\cite{Ka}*{Theorem~1.1},
thus $\Lambda_1$ is s primitive
sublattice of $\NS(JC)$ of rank three.

\subsubsection*{Case~3: $G= D_6$}
In this case, $\Aut(C)=\langle\tau,\rho\mid \rho^6,\tau^2,(\tau\rho)^2\rangle$,
where the hyperelliptic involution is
$\sigma=\rho^3$.
According to Lemma~\ref{ell}
the involution $\rho\tau\in\Aut(C)$ defines
the elliptic curve $E_{\rho\tau}$ in $JC$.
The intersection $E_\tau\cap E_{\rho\tau}$
consists of the origin plus the points $[p+\tau(p)-K_C]$,
where $p\in {\rm Fix}(\rho)$. The fixed points
of $\rho$ are contained in the Weierstrass 
points of $C$. 
By applying the Riemann-Hurwitz formula
to the degree six cyclic covering
$C\to C/\langle\rho\rangle\cong\pp^1$
we deduce that $\rho$ has no
fixed points. Thus $E_\tau\cdot E_{\rho\tau}=1$.
The intersection $E_{\sigma\tau}\cap E_{\rho\tau}$
consists of the origin plus the points $[p+\tau(p)-K_C]$,
where $p\in {\rm Fix}(\rho^2)$. 
By applying the Riemann-Hurwitz formula
to the degree three cyclic covering
$C\to C/\langle\rho^2\rangle\cong\pp^1$
we deduce that $\rho^2$ has four
fixed points exchanged in pairs by $\sigma\tau$.
Thus $E_{\sigma\tau}\cdot E_{\rho\tau}=3$.
The sublattice $\Lambda_2\subseteq\NS(JC)$ 
generated by the classes of $E_\tau$, 
$E_{\rho\tau}$ and $\Theta-2E_\tau-2E_{\rho\tau}$
is isometric to $U\oplus (-6)$.
Since $\det(\Lambda_2)$ is square free, 
we conclude that $\Lambda_2$ is s primitive
sublattice of $\NS(JC)$ of rank three.

\subsubsection*{Case~4: $G= 2D_6$}
In this case, $\Aut(C)=\langle \tau,\rho,\eta 
\mid \tau^2, \eta^2\rho^3, \eta^4, (\rho\tau)^2, 
(\tau\eta)^2, \allowbreak \eta^{-1}\rho\eta\rho\rangle$,
where $\sigma=\rho^3=\eta^2$ is the hyperelliptic involution.
Observe that $\langle\tau,\eta\rangle\cong D_4$
and $\langle\tau,\rho\rangle\cong D_6$.
Both $E_\tau$ and $E_{\sigma\tau}$ admit
complex multiplication of order three induced by 
the automorphism $\rho^3$, thus we obtain 
${\rm Hom} (E_\tau,E_{\sigma\tau})\cong \zz^2$,
so that $\NS(E_\tau\times E_{\sigma\tau})$ has rank four.
Let  $\Lambda_3\subset \NS(JC)$ be generated 
by the classes of $E_\tau$,  $E_{\rho\tau}$, $\Theta- E_{\eta\tau}$ and
$\Theta-2E_\tau-2E_{\rho\tau}$.
The only intersection number we need to compute 
is the one between $E_{\eta\tau}$ and $E_{\rho\tau}$ 
since the other ones appear in Case 2 and Case 3.
The intersection $E_{\eta\tau}\cap E_{\rho\tau}$
consists of the origin plus the points $[p+\eta\tau(p)-K_C]$,
where $p\in {\rm Fix}(\rho\eta^3)$. 
Observe that $(\rho\eta^3)^2=\sigma$, 
so that its fixed points are contained in the Weierstrass 
points of $C$. As in Case 2, we deduce that 
it has exactly two fixed points which are exchanged in pairs
by $\eta\tau$. Thus $E_{\eta\tau}\cdot E_{\rho\tau}=2$.
The lattice $\Lambda_3$ is isometric to $U\oplus (-2)\oplus (-6)$.
The only isotropic and non-zero vector in the discriminant 
group of $\Lambda_3$ is $\frac{1}{2}E_{\eta\tau}$.
This does not belong to $\NS(JC)$ by~\cite{Ka}*{Theorem~1.1},
thus $\Lambda_3=\NS(JC)$ since
$\Lambda_3$ is s primitive
sublattice of $\NS(JC)$ of the 
same rank.
Observe that an equation for $C$
in this case is $y^2=x^6-1$.

\subsubsection*{Case~5: $G=\tilde S_4$}
In this case, $\Aut(C)=\langle \tau,\rho,\eta 
\mid \tau^2, \eta^2\rho^3, \eta^4, (\rho\tau)^2, 
(\tau\eta)^2, \allowbreak (\rho\eta)^3\rangle$,
where $\sigma=\rho^3=\eta^2$
is the hyperelliptic involution.
As in the previous case
we only need to compute the intersection 
number between $E_{\eta\tau}$ and $E_{\rho\tau}$.
The intersection $E_{\eta\tau}\cap E_{\rho\tau}$
consists of the origin plus the points $[p+\eta\tau(p)-K_C]$,
where $p\in {\rm Fix}(\rho\eta^3)$.
In this case $(\rho\eta^3)^3=\sigma$, thus we deduce as 
in Case 3 that it has no fixed points.
Thus $E_{\eta\tau}\cdot E_{\rho\tau}=1$.
Let  $\Lambda_4\subset \NS(JC)$ be 
the lattice generated by the classes of 
$E_\tau$,  $E_{\rho\tau}$, 
$\Theta- E_\tau-E_{\eta\tau}$ and $E_\tau-E_{\eta\tau}+2E_{\rho\tau}$,
which is isometric to $U\oplus (-2)\oplus (-4)$.
Since the discriminant group of $\Lambda_4$
has no non-zero isotropic vectors, we conclude that 
$\Lambda_4=\NS(JC)$ since $\Lambda_4$ is s primitive
sublattice of $\NS(JC)$ of the same rank.
Observe that an equation for $C$
in this case is $y^2=x^5-x$.\\

To conclude our proof observe that 
since $\Lambda_3$ and $\Lambda_4$
are non-isometric lattices of rank four
and both of them are N\'eron-Severi lattices of
Jacobian of curves of the irreducible family given 
in Case~3, the very general curve 
$C$ of this irreducible family must have
$\NS(JC)$ of rank three. Hence we
deduce the equality $\Lambda_2=\NS(JC)$ for 
such a curve $C$ since $\Lambda_2$ is a 
primitive sublattice of $\NS(JC)$ of
the same rank. Similarly, we deduce 
the equality $\Lambda_1=\NS(JC)$ for the very 
general curve $C$ in Case~2.
Since $\Lambda_1$ and $\Lambda_2$ 
are non-isometric lattices and the family
of curves in Case~1 is irreducible, we
deduce by a similar argument the equality 
$\Lambda_0=\NS(JC)$ for the very 
general curve $C$ in Case~1.
\end{proof}

\begin{remark}
\label{vg}
Observe that the ranks of the N\'eron-Severi 
groups of $JC$ and 
$E_\tau\times E_{\sigma\tau}$ coincide, 
since the two abelian surfaces are isogenous.
The generality conditions on $C$ in Theorem~\ref{NS}
can be made explicit using the formula
\begin{equation}
\label{kani-rk}
{\rm rk}(\NS(JC))= {\rm rk}(\NS(E_\tau\times E_{\sigma\tau}))
 =
 {\rm rk}( {\rm Hom}(E_\tau,E_{\sigma\tau}))+2.
\end{equation}
For a proof of this formula see~\cite{Ka1}*{Proposition~22}.
Thus the rank of $\NS(JC)$ is two if and only if the two elliptic curves
$E_{\tau}$ and $E_{\sigma\tau}$ are non-isogenous
and the rank of $\NS(JC)$ is three if and only if the two elliptic curves
$E_{\tau}$ and $E_{\sigma\tau}$ are isogenous and
do not have complex multiplication.
\end{remark}

\section{Kummer surfaces} \label{se-ks}
Let $K$ be a field of characteristic different from two and denote by 
$\mathbb{A}^2_{t,s}$ the affine plane over $K$ with affine coordinates $t,s$.  
Let $\mathscr{C} \to \mathbb{A}^2_{t,s}$ be the projective family of genus
two curves given by the affine equation
\[
 \mathscr{C} \colon
 \quad
 y^2 = x(x-2t)(x-2s)\Big(x-\dfrac{1}{t}\Big)\Big(x-\dfrac{1}{s}\Big) .
\]
Setting $\gamma_1 = 2s^2-1$, $\gamma_2 = t-s$, $\gamma_3 = 2ts-1$ and $\gamma_4 = 2t^2-1$, we 
find that the family $\mathscr{C}$ is smooth over the open subset of $\mathbb{A}^2_{t,s}$ 
where $s t \gamma_1 \gamma_2 \gamma_3 \gamma_4 \neq 0$.  
Observe that the automorphism group of 
the family contains the hyperelliptic involution
$\sigma$ and an involution $\tau$ defined by
$(x,y)\mapsto \bigl( \frac{2}{x}, \frac{2 \sqrt{2} y}{x^3} \bigr)$.
Conversely, let $C$ be a smooth genus two curve whose automorphism
group $\Aut(C)$ contains the hyperelliptic involution
$\sigma$ and an involution $\tau\neq\sigma$.
The action of $\tau$ descends to an action of $\zz/2\zz
=\langle\bar\tau\rangle$
on the quotient curve $C/\langle\sigma\rangle\cong\pp^1$.
The automorphism $\tau$ does not fix any Weierstrass
point since the curve $C$ is smooth.
After a linear change of coordinates we can assume
that $\bar\tau$ is the map $x\mapsto\frac{2}{x}$
and that the images of the six Weierstrass points of
$C$ in $\pp^1$ are $\{0,\infty,2t,1/t,2s,1/s\}$ for some
$t,s\in\mathbb{C}$. 
Hence $C$ appears in the family $\mathscr{C}$.

\subsection{Klein form and automorphisms}
Let $\mathscr{X} \to \mathbb{A}^2_{t,s}$ be the relative minimal resolution 
of the quartic Kummer surface associated to the relative Jacobian $J\mathscr{C}$.  
The equations of $\mathscr{X}$ are 
\[
\mathscr{X} \colon \quad
\left\{
\begin{array}{r@{}r@{}r@{}r@{}rcr}
x_1^2 & - 2t^2x_2^2 & + \gamma_4 x_3^2 & & & = & 2 \gamma_4 x_0^2 \\[5pt]
& t^2sx_2^2 & - ts\gamma_2 x_3^2 & - ts^2x_4^2 & & = & - \gamma_2 x_0^2 \\[5pt]
& & \gamma_1 x_3^2 & -2s^2 x_4^2 & + x_5^2 & = & 2 \gamma_1 x_0^2 .
\end{array}
\right.
\]
These equations are obtained from the 
Klein equations for the resolution of the 
Kummer surface of the Jacobian of a genus two curve of 
the form $y^2=f_5(x)$ (see~\cite{PV}*{\S 5.3,~(36)}).
The surfaces in the family $\mathscr{X}$ contain exactly 
$32$ lines: they are cut out by the four hyperplane sections
with equations $x_1\pm x_5\pm 2 \gamma_2 x_0=0$.
Let $\ell \subset \mathscr{X}$ denote the family of lines with equations
\[
 \ell\colon\quad
 \left\{ 
 \begin{array}{rcl}
 x_2 & = & x_1 - \frac{1}{t} \gamma_4 x_0\\[3pt]
 x_3 & = & x_1 - 2t x_0\\[3pt]
 x_4 & = & x_1 - \frac{1}{s} \gamma_3 x_0\\[3pt]
 x_5 & = & x_1 - 2 \gamma_2 x_0 .
\end{array}
\right.
\]
For $(t,s) = (-1,1)$ the line $\ell$ contains 
the arithmetic progressions mentioned in the introduction.  
Denote by $[6]$ the set $\{0,\ldots,5\}$.  For each subset $\mathscr{P} \subset [6]$, denote by 
$\sigma_{\mathscr{P}}$ the automorphism of the family $\mathscr{X}$ that changes the sign 
of all the variables with indices in $\mathscr{P}$; clearly we have 
$\sigma_{\mathscr{P}} = \sigma_{[6] \setminus \mathscr{P}}$.  To simplify the notation, we sometimes 
denote the involution $\sigma_{\{a_1,\ldots,a_r\}}$ by $\sigma_{a_1 \cdots a_r}$.
It is straightforward to check that the 
lines corresponding to two distinct partitions $\mathscr{P} \sqcup \mathscr{Q}$ and 
$\mathscr{P}' \sqcup \mathscr{Q}'$ of $[6]$ intersect if and only if the automorphism 
$\sigma_{\mathscr{P}} \sigma_{\mathscr{P}'}$ is the sign change of a single variable.
The lines corresponding to the 16 two-set partitions with even size are pairwise non-adjacent, 
and similarly for the 16 partitions into sets of odd size.  Hence, the intersection graph of the 
lines on $\mathscr{X}$ is connected bipartite of type $(16,16)$, and it is regular of valence~6.
Observe that one set of pairwise non-adjacent lines is the set of the exceptional divisors 
of the resolution $X\to {\rm Kum}(JC)$, while the other set contains the images of 
the $16$ translates of a theta divisors in $JC$.

Let $\Sigma$ be the group of sign changes
of the coordinates: the group $\Sigma$ is isomorphic to $(\zz/2\zz)^5$ and
is a subgroup of the linear automorphism
group of any surface in the family $\mathscr{X}$.

\begin{proposition}\label{autofam}
Let $X\subset\pp^5$ be a fiber in the family
$\mathscr{X}$ and let $C$ be the corresponding
genus two curve.  The linear automorphism group of $X$
is isomorphic to a semidirect product of 
$\Aut(C)/\langle\sigma\rangle$ with $\Sigma$.
\end{proposition}

\begin{proof}
Let $G$ be the linear automorphism group of $X$.  Denote by 
$G_C$ the subgroup of $G$ induced by $\Aut(C)$: the group 
$G_C$ is isomorphic to $\Aut(C)/\langle\sigma\rangle$.  
We first show that the group $G$ is 
generated by $G_C$ and $\Sigma$.  
By what we have seen, the group $\Sigma$ acts simply 
transitively on the set of lines of $X$.
Let $\gamma$ be a linear automorphism of $X$.  
Composing $\gamma$ with an element of 
$\Sigma$ if necessary, 
we reduce to the case when $\gamma$ 
fixes the line $\ell$, in particular it preserves 
each of the two sets of pairwise disjoint lines of $X$.
We deduce that $\gamma$ comes
from an automorphism of $JC$.
Moreover, since it preserves the theta divisor 
corresponding to $\ell$, then it is induced 
by an automorphism of $C$ by the 
Torelli Theorem~\cite{La}*{Th\'eor\`eme~3}.

Each element of $G_C$ is a scalar
permutation by Theorem~\ref{aut} and 
fixes at least one of the lines. We deduce that 
the intersection $G_C \cap \Sigma$ consists of only the identity, and we are left 
to show that $\Sigma$ is a normal subgroup of the group $G$.
This follows from the fact that 
if $D\in\Sigma$ and $M\in G_C$,
then $M^{-1}DM$ is a diagonal involution
and hence is in $\Sigma$.
\end{proof}

\begin{remark}
The linear automorphism group of the
general fiber $X$ of $\mathscr{X}$ is
generated by $\Sigma$ and by the involution
\[
 \tau\colon X\to X
 \qquad
 [x_0,\dots,x_4]\mapsto
 [tsx_3 , 2t^2sx_2 , sx_1 , 2tsx_0 , tx_5 , 2ts^2x_4].
\]
In fact, the group $\Sigma$ is contained 
in the linear automorphism group of the 
resolution of any Jacobian Kummer surface 
defined by the Klein equations in $\pp^5$.
The involutions changing an even number of 
signs preserve each set of pairwise disjoint lines 
and do not preserve any line, this implies 
that they are induced by the translations by $2$-torsion 
points on $JC$.

If $X$ corresponds to a 
genus two curve $C$ with $\Aut(C)\cong D_4$,
then the linear automorphism group of
$X$ is generated by $\Sigma$, $\tau$ and
the involution $\rho$ defined in Subsection~\ref{auto}
\end{remark}

\subsection{Special families}
In this subsection we study the 
special loci in the affine plane $\mathbb{A}^2_{t,s}$
where the corresponding genus two curve admits
extra automorphisms. We compute the
Picard group of the minimal resolution of
the Kummer surface
over each component.

\begin{proof}[Proof of Theorem~\ref{Pic}]
In each case we determine the intersection 
matrix of $\Pic(X)$ as follows.
By means of Theorem~\ref{NS}
we first compute the transcendental lattice of $JC$
which is the orthogonal complement of $\NS(JC)$ 
in $H^2(JC,\zz)\cong U^{\oplus 3}$.
This easily gives the transcendental 
lattice $T(X)$, which is isometric to $T(JC)(2)$
by~\cite{Mo}*{Proposition~4.3}.
Finally, we obtain $\Pic(X)$ as the 
orthogonal complement of $T(X)$ in the K3 lattice.
Let $(A_L,q_L)$ denote the discriminant group of 
a sublattice $L$ of a unimodular lattice.
When computing the orthogonal complement 
of either $\NS(JC)$ or $T(X)$ we use the fact that 
the discriminant group of $A_{L^\perp}$ is isomorphic to 
$(A_{L},-q_L)$ and that,
under suitable conditions, a lattice 
with a given signature and discriminant 
quadratic form is unique up to isometries.
This uniqueness property follows 
from~\cite{N79}*{Theorem~1.13.2} 
and from the classification of positive 
definite binary quadratic forms \cite{CS}*{Table 15.1}.
Observe that in case $G=2D_6$ the quadratic form ${\rm diag}(2,6)$ 
is the only positive definite binary form with discriminant 
group isomorphic to $\zz/2\zz\oplus \zz/6\zz$ 
and having a non-zero isotropic vector.
In case $G=\tilde S_4$ the quadratic form ${\rm diag}(2,4)$ is the only 
positive definite binary form with discriminant 
group isomorphic to $\zz/2\zz\oplus \zz/4\zz$.

We will explain the case $G=D_4$, the other cases being similar.
The discriminant group of $\NS(JC)$ is isometric to $(\zz/2\zz)^{3}$ 
with quadratic form 
\[
q_{\NS(JC)}=\left(\begin{array}{cc} 
0 & \frac{1}{2} \\
\frac{1}{2} & 0
\end{array}
\right)\oplus \Big(-\frac{1}{2}\Big).
\]
Thus $T(JC)$ has signature $(2,1)$ and 
discriminant group  $((\zz/2\zz)^{3},-q_{\NS(JC)})$.
The only lattice with such property, 
up to isometries, is $U(2)\oplus (2)$.
Thus $T(X)$ is isometric to $U(4)\oplus (4)$
and its discriminant group
is $(\zz/4\zz)^3$ with quadratic form 
\[
 q_{T(X)} :=
  \left(
  \begin{matrix}
   0 & \frac{1}{4} \\
   \frac{1}{4} & 0
  \end{matrix}
  \right)\oplus \Big(\frac{1}{4}\Big).
 \]
The Picard lattice of $X$ has signature
$(1,18)$ and discriminant quadratic form 
$-q_{T(X)}$: again by the uniqueness property,
it is isometric to $U(4)\oplus(-4)\oplus E_8^{\oplus 2}$.
\end{proof}

Recall that the hyperelliptic involution
$\sigma$ in coordinates is $(x,y)\mapsto (x,-y)$,
while the involution $\tau$ is 
$(x,y) \mapsto \bigl( \frac{2}{x}, \frac{2\sqrt{2}y}{x^3} \bigr)$.
The corresponding families of quotient curves 
are smooth of genus one and 
admit the following affine equations:
\begin{align*}
 E_1 := \mathscr{C}_{t,s}/\langle\tau\rangle
 \qquad
 y^2 = (x+2\sqrt{2})\Big(x-2t-\dfrac{1}{t}\Big)\Big(x-2s-\dfrac{1}{s}\Big) , \\
 E_2 := \mathscr{C}_{t,s}/\langle\sigma\tau\rangle
 \qquad
 y^2 = (x-2\sqrt{2})\Big(x-2t-\dfrac{1}{t}\Big)\Big(x-2s-\dfrac{1}{s}\Big) .
\end{align*}
The quotient maps are
$(x,y)\mapsto \bigl( x+\frac{2}{x}, \frac{y}{x} \pm \frac{\sqrt{2} y}{x^2} \bigr)$,
where the positive sign is for $E_1$
and the negative sign is for $E_2$.
For special values of $t$ and $s$
the hyperelliptic curve curve $\mathscr{C}_{t,s}$
acquires more automorphisms: see Table~\ref{tab:auts} and also~\cite{CGLR}.

\begin{table}[h]
\begin{tabular}{l|l}
$\Aut(C)$  & Locus in the $(t,s)$ plane\\
\midrule
$D_4$ & $(t+s) (2 ts + 1) = 0$ \\[5pt]
$D_6$ & $(t^{2} - ts + s^{2} - \frac{1}{2}) (4t^2s^2 - 2ts - 2t^2 + 1)$ \\[5pt]
& $(4t^2s^2 - 2ts - 2s^2 + 1) ( 2t^2s^2 - t^2 + ts - s^2) = 0$ \\[5pt]
$2D_6$ & 
$\pm \left( -\sqrt{\frac{3}{2}},\frac{1}{\sqrt{6}} \right) , \; \pm \left( \frac{1}{\sqrt{6}},-\frac{1}{\sqrt{6}} \right) , \; \pm \left( \sqrt{\frac{3}{2}},-\sqrt{\frac{3}{2}} \right) , \; \pm \left( \frac{1}{\sqrt{6}},-\sqrt{\frac{3}{2}} \right) $ \\[7pt]
$\tilde{S}_4$ & 
$\pm \left( \dfrac{i-1}{2} , \dfrac{1-i}{2} \right) , \; \pm \left( \dfrac{i-1}{2} , \dfrac{1+i}{2} \right) , \; \pm \left( \dfrac{i+1}{2} , \dfrac{-1-i}{2} \right) , \; \pm \left( \dfrac{i+1}{2} , \dfrac{i-1}{2} \right)$
\end{tabular}
\vspace{5pt}
\caption{Automorphism groups of the curves in the family $\mathscr{C}_{t,s}$} \label{tab:auts}
\end{table}

\noindent
The curves in the affine $(t,s)$-plane
corresponding to $D_4$ and $D_6$
intersect in $17$ points.
One of them is $(0,0)$ which corresponds to a singular
curve. The remaining sixteen points
appear in Table~\ref{tab:auts}.
The eight points for $2D_6$ correspond to
isomorphic curves of genus two
and the same holds for the eight points
in the case~$\tilde S_4$.

\begin{example}
\label{bre}
In~\cite{Br} A.~Bremner 
investigated the geometry 
of the K3 surface defined by 
\[
X_{B} \colon \quad 
\left\{
\begin{array}{rcl}
x^2+xu -u^2 & = & w^2 +wz -z^2 \\[5pt]
y^2+yv -v^2 & = & u^2 +ux-x^2 \\[5pt]
z^2 +zw-w^2 & = & v^2 +vy -y^2.
\end{array}
\right.
\]
His main motivation was the search
for rational points of $X_B$, 
since these give (cubing and summing the equations) 
rational solutions to the classical
Diophantine equation
$x^6+y^6+z^6=u^6+v^6+w^6$.
Recently M. Kuwata \cite{Ku}*{Theorem~5.1, Corollary~5.2} proved 
that $X_B$ is the Kummer surface of 
$JC_B$, where   
\[
C_B \colon \quad y^2=(x^2+1)(x^2+2x+5)(x^2-2x+5),
\]
and that $\Pic(X_B)$ has rank $19$.
An easy computation shows that 
$C_B=\mathscr{C}_{t,s}$ with
$(t,s)=(\frac{1}{\sqrt{-2}}-\sqrt{2},\sqrt{-2})$,
so that its automorphism group is isomorphic to 
$D_{6}$. Since the $j$-invariants of 
the quotient curves $E_1$, $E_2$
are not rational integers,  
the curves do not admit complex multiplication 
so that ${\rm rk\, Hom}(E_1,E_2)=1$.
By Theorem \ref{Pic}  
the Picard lattice of $X_B$ is isometric to 
$U(2)\oplus E_8^2\oplus (-12)$.
We observe that a basis of $\Pic(X_B)$ 
had been given in~\cite{Br}*{Theorem~3.3}.
\end{example}

\begin{example}
\label{bb}
In~\cite{BB} J.~Browkin and J.~Brzezi{\'n}ski
study the solutions of the system 
\[
X_{BB} \colon \quad
\left\{
\begin{array}{rcl}
 x_1^2-2x_2^2+x_3^2 & = & x_2^2-2x_3^2+x_4^2 \\[5pt]
 x_1^2-2x_2^2+x_3^2 & = & x_3^2-2x_4^2+x_5^2 \\[5pt]
 x_1^2-2x_2^2+x_3^2 & = & x_4^2-2x_5^2+x_6^2
\end{array}
\right.
\]
consisting of the sequences of six integer squares
with constant second differences.
The surface $X_{BB}$ is a smooth K3 surface.
An easy computation shows that it is the Kummer 
surface of the genus two curve $\mathscr C_{t,s}$
with $(t,s)=(2\sqrt{2},\frac{\sqrt{2}}{3})$.
The automorphism group of the curve
is $V_4$. Moreover the $j$-invariants of
the two curves $E_1$ and $E_2$ are 
$\frac{2^6\cdot 7^3\cdot 97^3}{3^6\cdot 5^4}$ 
and $\frac{2^6\cdot 7^3}{3^2}$
respectively. Since the prime $5$ is of bad 
reduction just for $E_1$, the curves
$E_1$ and $E_2$ are not isogenous
and thus ${\rm rk\, Hom}(E_1,E_2)=0$.
By Theorem~\ref{Pic} the Picard lattice 
of $X_{BB}$ is isometric to 
$U\oplus E_8\oplus D_7\oplus (-4)$.
\end{example}

\begin{remark}
Let $X$ be any smooth surface in 
the family $\mathscr{X}$.
The surface $X$ is a moduli space of rank two vector bundles on itself.
This gives a natural modular interpretation to the rational points on $X$.  
We do not know if there is a similar interpretation for the integral points on $X$.

Mukai studied the moduli spaces of stable vector bundles on abelian and K3 surfaces in its influential paper~\cite{mukai}.  In particular, he showed in~\cite{mukai}*{Example~0.9} the existence of a beautiful correspondence between the K3 surfaces that are intersections of three quadrics in $\mathbb{P}^5$ and certain moduli spaces of vector bundles on the K3 surfaces themselves.  The correspondence is obtained as follows.  Let $Y$ be a smooth complete intersection of three quadrics in $\mathbb{P}^5$; denote by $\mathbb{P}^2$ the net of quadrics containing $X$.  In the net of quadrics $\mathbb{P}^2$, the discriminant locus $\Delta_Y$ is the locus of quadrics of rank at most $5$: the discriminant $\Delta_Y$ is a plane curve of degree six.  Denote by $M_Y$ minimal resolution of the double cover of $\mathbb{P}^2$ branched over the sextic $\Delta_Y$.  The general point of $M_Y$ corresponds to a quadric $Q$ containing $Y$, together with the choice of a ruling by $2$-planes of the quadric $Q$.  Identifying smooth quadrics in $\mathbb{P}^5$ with the Grassmannian $Gr(2,4)$, we see that the two rulings by $2$-planes of a quadric $Q$ correspond to the second Chern classes of the rank $2$ vector bundles on $Q$ determined by the dual of the tautological vector bundle and by the universal quotient bundle.  Thus, to each point of $M_Y$ corresponding to a smooth quadric $Q$ containing $X$ together with a choice of ruling on $Q$, we associate a rank two vector bundle on $X$ by restricting to $X$ itself the rank two vector bundle on $Q$ determined by the chosen ruling.  This correspondence induces an isomorphism between $M_Y$ and the irreducible component of the moduli space $\mathcal{M}(2,\mathcal{O}_X(1),2)$ of stable rank two vector bundles $E$ on $X$ with Chern classes $c_1(E) = [\mathcal{O}_X(1)]$ and $\deg(c_2(E)) = 4$.

It is classically known that the K3 surfaces that are complete intersections of a net of quadrics and the double cover of the net of quadrics branched over the discriminant locus are isomorphic as soon as the surface contains a line;
this applies to our case as $X$ contains
$32$ lines.
A modern reference for this result is~\cite{MN}.

\end{remark}

\subsection{Elliptic fibrations}

\begin{theorem} \label{thm:quad}
Let $X$ be a smooth surface in the
family $\mathscr{X}$; 
the following statements hold.
\begin{enumerate}
\item \label{it:qr3}
There are no quadrics of rank three containing $X$.
\item \label{it:qr4}
The quadrics of rank four containing $X$ appear in Table~\ref{rg4}.
\item \label{it:ef}
Every quadric of rank four containing $X$ determines two elliptic fibrations on $X$ defined by pencils of hyperplane sections.
\end{enumerate}
\end{theorem}

{\footnotesize
\begin{longtable}{c@{\;\;\;}@{\;\;\;}c@{\;\;\;}@{\;\;\;}c}
\hline \\[-7pt]
 $\gamma_1(x_1^2 -2t^2 x_2^2) + \gamma_4 (2s^2 x_4^2 - x_5^2)$
& $sx_1^2 +\gamma_2x_3^2 - tx_5^2 -4ts\gamma_2x_0^2$ 
& $t (ts x_2^2 - s \gamma_2 x_3^2 - s^2 x_4^2) + \gamma_2x_0^2$ \\
\\[-7pt] \hline \\[-7pt]
 $ x_1^2 - 2t^2x_2^2 + \gamma_4 (x_3^2 -2 x_0^2)$ 
& $ts(\gamma_2x_1^2 + t\gamma_3x_2^2 -s\gamma_4x_4^2) -\gamma_2\gamma_3\gamma_4x_0^2$ 
& $s\gamma_3x_1^2 + 4t^3s\gamma_2x_2^2 -\gamma_2\gamma_3\gamma_4 x_3^2 -t\gamma_4x_5^2$ \\
\\[-7pt] \hline \\[-7pt]
 $\gamma_1x_3^2 - 2s^2x_4^2 + x_5^2 -2\gamma_1x_0^2$ 
& $ ts(t\gamma_1 x_2^2 -s\gamma_3 x_4^2 + \gamma_2 x_5^2) -\gamma_1\gamma_2\gamma_3x_0^2$ 
& $s\gamma_1x_1^2 -\gamma_1\gamma_2\gamma_3 x_3^2 + 4ts^3\gamma_2x_4^2 -t\gamma_3x_5^2$ \\
\\[-7pt] \hline \\[-7pt]
 $s(x_1^2 + \gamma_3x_3^2 - 2ts x_4^2) -2t\gamma_3x_0^2$ 
& $t(\gamma_3x_1^2 + 2t\gamma_2x_2^2 -\gamma_4x_5^2) -2\gamma_2\gamma_3\gamma_4x_0^2$ 
& $\gamma_2x_1^2 +2t^3\gamma_3x_2^2 -\gamma_2\gamma_3\gamma_4 x_3^2 -2ts^2\gamma_4x_4^2$ \\
\\[-7pt] \hline \\[-7pt]
 $t(2tsx_2^2 -\gamma_3x_3^2 - x_5^2) + 2s\gamma_3x_0^2$ 
& $s(\gamma_1x_1^2 + 2s\gamma_2x_4^2 -\gamma_3x_5^2) -2\gamma_1\gamma_2\gamma_3x_0^2$ 
& $2st^2\gamma_1x_2^2 - \gamma_1\gamma_2\gamma_3 x_3^2 -2s^3\gamma_3x_4^2 + \gamma_2x_5^2$ 
\\[2pt] \hline
\\[-5pt]
\caption{Quadrics of rank four containing $X$} \label{rg4}
\end{longtable}}

\begin{proof}
\eqref{it:qr3}
Suppose that $Q$ is a quadric of rank at most three containing $X$, and hence the singular locus of $Q$ contains a two-dimensional plane $\Pi$.  Let $Q_1,Q_2$ be two quadrics such that $X = Q \cap Q_1 \cap Q_2$; the intersection $X \cap \Pi$ contains $\Pi \cap Q_1 \cap Q_2$ and is therefore not empty and consists of singular points of $X$, contradicting the assumption that $X$ is smooth.  We deduce that every quadric vanishing on $X$ has rank at least four.

\eqref{it:qr4}
Since the surface $X$ is defined by diagonal quadrics, it follows that the quadrics of rank at most four containing $X$ are obtained from the equations of $X$ by eliminating two or more of the variables.  Eliminating a variable corresponds to a linear equation on the net of quadrics containing $X$, so that for each pair $\mathscr{S}$ of variables there is a quadric containing $X$ and not involving the variables in $\mathscr{S}$.  An easy computation shows that these are exactly the quadrics defined by the polynomials in Table~\ref{rg4}.

\eqref{it:ef}
Let $Q \subset \mathbb{P}^5$ be a quadric of rank four containing the surface $X$ (see Table~\ref{rg4}); since $X$ is smooth and it is a complete intersection of quadrics, it follows that $X$ does not meet the vertex of $Q$.  In particular, the rational map $\mathbb{P}^5 \dashrightarrow \mathbb{P}^3$ obtained by projecting away from the singular subscheme of $Q$ induces a morphism $\pi_Q \colon X \to \overline{Q}$, where $\overline{Q} \subset \mathbb{P}^3$ is the image of the projection of $Q$ from its vertex; composing the morphism $\pi_Q$ with one of the two projections $\overline{Q} \simeq \mathbb{P}^1 \times \mathbb{P}^1 \to \mathbb{P}^1$ we obtain a fibration $X \to \mathbb{P}^1$.  It is immediate to check that these fibrations have connected fibers and therefore they are elliptic fibrations: we deduce from items~\eqref{it:qr3} and~\eqref{it:qr4} that there are 30 elliptic fibrations on $X$ arising in this way.
\end{proof}

\begin{theorem}
\label{dense}
The family of Kummer surfaces $\mathscr{X} \to \mathbb{A}^2_{t,s}$ admits 
the structure of a family of elliptic fibrations over $\mathbb{P}^1$ with Mordell-Weil 
group of positive rank.  In particular, if $X$ is a smooth surface in the 
family $\mathscr{X}$ defined over a field $K$, then the set of $K$-rational points of $X$ is Zariski-dense.
\end{theorem}
\begin{proof}
The family $\mathscr{X}$ is contained in the 
rank four quadric $\mathscr{Q}$ of equation
\[
\mathscr{Q} \colon \quad \gamma_1 (x_1^2 - 2t^2 x_2^2) + 
 \gamma_4 (2s^2x_4^2-x_5^2) = 0.
\]
This quadric is a cone over 
$\pp^1\times\pp^1 \times \mathbb{A}^2_{t,s}$ with vertex a line.
The two rulings of $\pp^1\times\pp^1$ induce
two elliptic fibrations $\pi_{\pm} \colon \mathscr{X} \to \pp^1 \times \mathbb{A}^2_{t,s}$.
We partition the set $\mathscr{L}$ of lines of $\mathscr{X}$ into the sets 
\[
\mathscr{L}_+ = 
\begin{pmatrix}
\ell & \sigma_0 \ell & \sigma_3 \ell & \sigma_{03} \ell \\
\sigma_{14} \ell & \sigma_{014} \ell & \sigma_{134} \ell & \sigma_{25} \ell \\
\sigma_{24} \ell & \sigma_{024} \ell & \sigma_{234} \ell & \sigma_{15} \ell \\
\sigma_{45} \ell & \sigma_{045} \ell & \sigma_{345} \ell & \sigma_{12} \ell 
\end{pmatrix} ,
\quad 
\mathscr{L}_- = 
\begin{pmatrix}
\sigma_1 \ell & \sigma_{01} \ell & \sigma_{13} \ell & \sigma_{013} \ell \\
\sigma_{4} \ell & \sigma_{04} \ell & \sigma_{34} \ell & \sigma_{034} \ell \\
\sigma_{124} \ell & \sigma_{35} \ell & \sigma_{05} \ell & \sigma_{5} \ell \\
\sigma_{145} \ell & \sigma_{23} \ell & \sigma_{02} \ell & \sigma_{2} \ell 
\end{pmatrix} ;
\]
each row of $\mathscr{L}_+$ forms a square in $\mathscr{X}$ and similarly for the rows of $\mathscr{L}_-$.
The lines of $\mathscr{L}_+$ form four fibers of 
$\pi_+$ of type $I_4$ and are sections of $\pi_-$; 
similarly, the lines of $\mathscr{L}_-$ form four fibers of 
$\pi_-$ of type $I_4$ and are sections of $\pi_+$.
Thus both elliptic fibrations have Mordell-Weil 
group of positive rank, since they
admit two intersecting sections~\cite{Shioda}*{Lem.~8.2 and Thm.~8.4}.

The last statement follows at once since all the lines of 
$X$ are defined over the field $K$.
\end{proof}

\section{The B\"uchi K3-surface} \label{seku}
From now on we denote by $X$ the B\"uchi K3 
surface $X_5$ or equivalently $\mathscr{X}_{-1,1}$.
In particular $X$ is the Kummer
surface of the genus two curve $\mathscr C_{t,s}$
with $(t,s)=(-1,1)$ which has affine equation 
$y^2 = x(x^2-4)(x^2-1)$.
Thus $\Aut(C)$ is isomorphic to $D_4$.
Since the $j$-invariants of 
the quotient curves $E_1$, $E_2$
are not rational integers, then 
the curves do not admit complex multiplication 
so that ${\rm rk\, Hom}(E_1,E_2)=1$.
By Theorem \ref{Pic} we conclude
\[
 \Pic(X)
 \cong
 U(4)\oplus(-4)\oplus E_8^{\oplus 2}.
\]
Over the field $\mathbb{Q}(\sqrt{-6})$ we can 
absorb more of the coefficients of the equation 
of the Kummer surface $S$ in $\mathbb{P}^3$ 
to obtain the equation 
\[
 9 (x_0^4 + x_1^4 + x_2^4 + x_3^4) 
 + 6(x_0^2x_1^2 + x_2^2x_3^2) +
     6 (x_0^2x_2^2 + x_1^2x_3^2)
 - 14 (x_0^2x_3^2 + x_1^2x_2^2) 
  = 0.
\]

\subsection{Elliptic fibrations} \label{seell}

In this section we determine the elliptic fibrations 
on the surface $X$ induced by the quadrics of 
rank four containing $X$. Such quadrics are
given in Table~\ref{rg4} for $\mathscr{X}_{t,s}$.
We describe here the corresponding singular 
fibers and the Mordell-Weil groups for the
B\"uchi K3 surface $X$.
As a consequence, we find all the conics contained in~$X$.
Observe that each of the 15 quadrics in Table~\ref{rg4} 
does not involve a pair of variables of $\mathbb{P}^5$.  
In the three columns of Table~\ref{rg4} the first quadric 
is invariant under the linear automorphism group, while 
the remaining four form two pairs of orbits.

We later analyze in detail the fibrations associated to the 
first quadric in the first column of Table~\ref{rg4}.  
In general, using the Shioda-Tate formula (\cite{ST}*{Theorem~1.1}) 
and an explicit set of generators of the Picard group of $X$,
described in Proposition~\ref{pic}, we computed with 
Magma~\cite{Magma} the Mordell-Weil groups of the 
fibrations associated to the quadrics in Table~\ref{rg4}: 
\begin{enumerate}
\item \label{2+4}
the Mordell-Weil groups associated to the first quadric in the first column are isomorphic to $\mathbb{Z} \oplus (\mathbb{Z}/2\mathbb{Z} \oplus \mathbb{Z}/4\mathbb{Z})$, with singular fibers $4I_4+4I_2$; 
\item \label{4+0}
the Mordell-Weil groups associated to the remaining quadrics in the first column are isomorphic to $\mathbb{Z}^3 \oplus (\mathbb{Z}/4\mathbb{Z})$, with singular fibers $4I_4+2I_2+4I_1$; 
\item \label{2+0}
the Mordell-Weil groups associated to the quadrics in the last two columns are isomorphic to $\mathbb{Z}^5 \oplus (\mathbb{Z}/2\mathbb{Z})$, with singular fibers $4I_4+8I_1$.
\end{enumerate}
Each fibration contains four fibers of type $I_4$, represented by four squares consisting of lines.  We also checked that the torsion elements of the Mordell-Weil group are represented by lines on $X$ and that the quotient of the Mordell-Weil group by the subgroup generated by the lines is free of rank $0$ in case~\eqref{2+4}, rank $1$ in case~\eqref{4+0}, and rank $2$ in case~\eqref{2+0}.

\begin{proposition} \label{con4}
Every conic on $X$ is contained in a reducible fiber of an elliptic pencil on $X$, and viceversa, every fiber of type $I_2$ in an elliptic pencil on $X$ arising from the quadrics in Table~\ref{rg4} is the union of two conics.
\end{proposition}

\begin{proof}
Let $C \subset X$ be a conic on $X$ and denote by $\Pi \subset \mathbb{P}^5$ the two-dimensional plane containing $C$; it suffices to show that there is a quadric of rank four containing $X$ and the plane $\Pi$.  Restricting the equations of $X$ to the plane $\Pi$ we find that they are all proportional to an equation defining the conic $C$.  We obtain that there is a pencil $\mathscr{P}$ of quadrics containing $X$ and the plane $\Pi$.  It follows that the base locus of the pencil of quadrics $\mathscr{P}$ is singular and hence that there is a quadric $Q$ in the pencil $\mathscr{P}$ of rank at most four (see for instance~\cite{Don}*{Subsection~1.2}).  By Theorem~\ref{thm:quad}\eqref{it:qr3} the rank of the quadric $Q$ is four and we conclude by Theorem~\ref{thm:quad}\eqref{it:ef}.

The converse follows from the explicit analysis of the elliptic fibrations arising from the quadrics of rank four of Table~\ref{rg4}.
\end{proof}

We now analyze the elliptic fibrations on $X$ associated to the first quadric $x_1^2 -2x_2^2 +2x_4^2-x_5^2 = 0$ in Table~\ref{rg4}: we concentrate on these fibrations, since the Mordell-Weil group in these cases is generated by the lines on $X$.  
The two corresponding elliptic fibrations are defined by:
\[
 \pi_{\pm} \colon X\to\pp^1
 \qquad
 [x_0,\dots,x_5]
 \mapsto
 [x_2\pm x_4,x_1+x_5].
\]
Observe that linear automorphisms of $X$ permutes these two fibrations.
For example, $[x_2-x_4,x_1+x_5]$
and $[x_2+x_4,x_1-x_5]$ 
give fibrations defining the 
same elliptic pencil on $X$
due to the following relation
in $I(X)$:
\[
 (x_1+x_5)(x_1-x_5)
 =
 2(x_2+x_4)(x_2-x_4).
\]
The types of the singular fibers of two fibrations 
$\pi_+$ and $\pi_-$ coincide (since $\pi_+$ and 
$\pi_-$ are exchanged by an automorphism 
of the surface) and they are 
\begin{itemize}
\item 
$4$ fibers of type $I_4$, whose 
components are the trivial lines of $X$ lying above the points 
$[1,1] , [-1,1] , [2,1] , [-2,1]$;
\item 
$2$ fibers of type $I_2$, whose 
components are conics defined over 
$\qq(\sqrt{2})$ lying above the points 
$[\sqrt{2},1] , [-\sqrt{2},1]$; and 
\item 
$2$ fibers of type $I_2$ whose components 
are conics defined over $\qq(\sqrt{-2})$ lying above the points 
$[\sqrt{-2},1] , [-\sqrt{-2},1]$.
\end{itemize}

\subsection{The Picard group} \label{subs:baspic}

Let $\zeta_8$ be a primitive complex $8$-th
root of unity.
The Galois group $G$ of the extension
$\qq(\zeta_8)/\qq$ acts on $X$ and on
its Picard lattice.
Let $C_1$ and $C_2$ be the
two conics of equations:
\[
 C_1\colon\quad
 \left\{
 \begin{array}{rcl}
  2x_2^2 - x_5^2 + 12x_0^2 & = & 0\\
  x_1  +\sqrt{2}\,x_2 & = & 0\\
  x_4+\sqrt{2}\,x_5 & = & 0\\
  2x_3 -\sqrt{2}\, x_0 & = & 0
 \end{array}
 \right.
 \qquad
 C_2\colon\quad
 \left\{
 \begin{array}{rcl}
  2x_4^2 - x_5^2 + 4x_0^2 & = & 0\\
  x_1  -\sqrt{-2}\,x_4 & = & 0\\
  2x_2-\sqrt{-2}\,x_5 & = & 0\\
  x_3 +\sqrt{-2}\, x_0 & = & 0
 \end{array}
 \right.
\]
which are contracted by $\pi_-$
and are defined over $\qq(\sqrt{2})$ 
and $\qq(\sqrt{-2})$ respectively.

\begin{proposition}\label{pic}
Let $\Pic(X)_G$ be the subgroup
of $\Pic(X)$ fixed by the Galois group $G$. 
There is a decomposition of $G$-modules
\[
 \Pic(X)
 =
 \Pic(X)_G\oplus\langle C_1,C_2\rangle,
\]
where $\Pic(X)_G$ has rank $17$ and is
spanned by the classes of the lines of $X$.
\end{proposition}
\begin{proof}
Since any line of $X$ is defined over the
rationals, the classes of the $32$ 
lines of $X$ span a sublattice $\Lambda$
of $\Pic(X)_G$. By calculating the
Smith form of the intersection matrix
of these lines plus $C_1$ and $C_2$
one deduces that the discriminant
of the lattice $\Lambda\oplus\langle C_1,C_2\rangle$
is $2^6$, which is also the discriminant
of $\Pic(X)$ by Theorem~\ref{Pic}.
Thus $\Pic(X)=\Lambda\oplus\langle C_1,C_2\rangle$
and the statement follows.
\end{proof}

\subsection{Rational points} \label{sub:ratpt}
We apply the results of the previous sections
to the study of the rational points on $X$.
Several authors used elliptic curves of positive rank 
to prove Zariski density of rational points on elliptic 
surfaces, see for instance~\cites{BT,HT}.
We can use the results of
the previous section to 
explicitly determine an element
of infinite order on the
generic fiber $X_\eta$ of $\pi_-$.
By putting $\alpha := \frac{x_4-x_2}{x_1+x_5}$
we get that $X_\eta$ is 
a quartic curve contained in the 
three dimensional linear subspace
of $\pp^5$.
We chose an origin for the Mordell-Weil
group of $X_\eta$ to be the point $O$ 
whose coordinates are:
\begin{align*}
 O := [-2\alpha - 2 ,2\alpha + 1 , -2\alpha , -2\alpha + 1 , 2\alpha - 2 , 1].
\end{align*}
The point $O$ corresponds to a line 
which is a section of the fibration $\pi_-$.
Now we take the point $Q$ which
corresponds to the line $R_2$.
Since $R_1$ and $R_2$ intersect
at one point, then by~\cite{Shioda}
the point $Q$ has infinite order
with respect to the origin $O$.
Moreover we know $Q$ to be
the generator of the free part
of the Mordell-Weil group of $X_\eta$.
Its coordinates are the following:
\begin{align*}
 Q := [2\alpha + 2 , -2\alpha - 1 , 2\alpha , 2\alpha - 1 , -2\alpha + 2 ,1].
\end{align*}
By calculating $2Q$ we obtain the following result.

\begin{proposition}
A one-parameter family
of non-trivial rational solutions
of the B\"uchi problem is 
\begin{align*}
 x_0 & = 12\alpha^4 + 5\alpha^2 - 1
 & x_3 & = 8\alpha^5 - 10\alpha^3 - 6\alpha\\
 x_1 & = 8\alpha^5 + 8\alpha^4 + 22\alpha^3 - 2\alpha^2 + 2\alpha + 2
 & x_4 & = 8\alpha^5 - 4\alpha^4 - 2\alpha^3 - 11\alpha^2 + 2\alpha - 1\\
 x_2 & = -8\alpha^5 - 4\alpha^4 + 2\alpha^3 - 11\alpha^2 - 2\alpha - 1
 & x_5 & = -8\alpha^5 + 8\alpha^4 - 22\alpha^3 - 2\alpha^2 - 2\alpha + 2.
\end{align*}
\end{proposition}

\subsection{Supersingular primes} \label{sesu}

In this subsection we will
prove a congruence property for
the supersingular primes for the 
elliptic curve $E$ of equation
$y^2=x^3-8x^2-2x$, isomorphic over 
the rational numbers to the 
elliptic curve of equation
$y^2=x^3-35x^2+35x-1$.

\begin{lemma}\label{24}
The number field $K_4$ 
obtained by extending $\qq$
with the abscissas of all the 
$4$-torsion points of the curve 
$E$ is the $24$th cyclotomic 
field $\qq(\mu_{24})$.
\end{lemma}
\begin{proof}
The field $K_4$ is the
splitting field of the polynomial
\[
 x(x^2-8x-2)(x^2+2)
 (x^4 - 16x^3 - 12x^2 + 32x + 4).
\]
The roots of the degree four factor are 
\[
4+2\sqrt{3}\pm (2\sqrt{6}+ 3\sqrt{2}) \quad \quad {\textup{and}} \quad \quad 4-2\sqrt{3}\pm (2\sqrt{6}-3\sqrt{2}) ,
\]
hence an easy calculation shows
$\qq(\mu_{24})=K_4$, as required.
\end{proof}

A calculation using Magma shows that the abelianization of the Galois group of the field $\mathbb{Q}(E[8])$ is the $24$th cyclotomic field, but we do not need this more precise information: for us it is enough to know that the field $\mathbb{Q}(\mu_{24})$ is contained in~$\mathbb{Q}(E[8])$.

In the next theorem we adapt a strategy that we learned from~\cite{Bra}.  In the proof we work with the $8$-torsion subgroup, since the information coming from the $4$-torsion elements is not enough for our purposes.

\begin{theorem} \label{susie}
Let $p$ be a supersingular prime
for the elliptic curve $E$. Then 
$p\equiv 5$ or $23\mod 24$.
\end{theorem}
\begin{proof}
We concentrate on the Galois 
representation
\[
 \rho \colon
 \Gal(\bar\qq/\qq)
 \to
 \GL_2(\zz/8\zz)
\]
obtained by acting with the
absolute Galois group on the set
of $8$-torsion points of the curve $E$.
By Lemma~\ref{24} the image $G(8)$ of 
$\rho$ admits a quotient $A(8)$ isomorphic 
to the Galois group of the extension
$\qq(\mu_{24})/\qq$.
Given a prime of good reduction $p$, the Frobenius 
automorphism acting on the 
reduction of $X$ modulo $p$ lifts
to an element $\Phi_p\in\Gal(\bar\qq/\qq)$.
The image of $\Phi_p$ in the quotient $A(8)\cong(\zz/24\zz)^*$ 
of the representation $\rho$ is a 
congruence class modulo $24$.
If $p$ is a supersingular prime
for $E$, then $\Phi_p$ has null trace.
We will show that the image in $A(8)$
of elements with null trace is 
congruent to either $5$ or $23$
modulo $24$,
concluding the proof.
To prove the last step, we
made use of Magma to show
that the image $G(8)$ of the
Galois representation $\rho$
is generated by the following
matrices:
\[
 \left(
 \begin{matrix}
  4 & 3\\
  1 & 4
 \end{matrix}
 \right)
 \qquad
 \left(
 \begin{matrix}
  2 & 1\\
  1 & 2
 \end{matrix}
 \right)
 \qquad
 \left(
 \begin{matrix}
  3 & 6\\
  2 & 7
 \end{matrix}
 \right)
 \qquad
 \left(
 \begin{matrix}
  1 & 2\\
  6 & 3
 \end{matrix}
 \right).
\]
The group $G(8)$ is a solvable
group of order $2^6$ and index 
$2^3\cdot 3$ in $\GL_2(\zz/8\zz)$.
It contains $16$ elements of null
trace whose images in the quotient
$A(8)$ form a subset $S$ of cardinality
$2$. Since the primes $5$ and $167$
are supersingular for $E$ and
the second is congruent to $23$ 
modulo $24$, the images of
$\Phi_5$ and $\Phi_{167}$ in $A(8)$
exhaust the whole $S$. 
This concludes the proof.
\end{proof}

\begin{theorem}
Let $p$ be a supersingular
prime for the elliptic curve $E$
such that $p\equiv 23\mod 24$.
Then the number of points
of the reduction of $X$ modulo $p$ equals
$p^2+18p+1$.
\end{theorem}
\begin{proof}
Let $JC$ be the Jacobian of $C$
and let $Y:=JC/\langle\sigma\rangle\subset\pp^3$
be its quotient with respect to the
involution $\sigma \colon w\mapsto -w$.
Observe that $\# X(\ff_p)=\# Y(\ff_p)+16p$
for any prime $p\geq 5$ since all the sixteen singular 
points of $Y$ are defined over $\zz$.
Consider two subsets of $JC$ over $\ff_{p^2}$:
\[
 S_1:=\{w\in JC(\ff_{p^2}): \Phi_p(w)=w\}
 \qquad
 S_2:=\{w\in JC(\ff_{p^2}):\Phi_p(w)=-w\},
\]
where $\Phi_p$ is the Frobenius map.
We have $\# Y(\ff_p) = \frac{1}{2}(\# S_1 + \# S_2)$.
We now want to show that $S_1=JC(\ff_p)$
and that
\begin{equation}\label{S2}
 \# S_2 = \#J\tilde{C}(\ff_p),
\end{equation}
where $\tilde{C}$ is the quadratic
twist of $C$. The first equality is obvious.
About the second equality, let $y^2=f(x)$
be an equation for $C$ and $dy^2=f(x)$
be an equation for $\tilde{C}$. Consider
the isomorphism
\[
 f \colon C\to\tilde{C}
 \qquad
 (x,y)\mapsto \left (x , \frac{1}{\sqrt{d}} y \right)
\]
defined over $\ff_{p^2}$ and observe that
$\Phi_p\circ f = \tilde{\iota}\circ f\circ\Phi_p$,
where $\tilde{\iota}$ is the hyperelliptic
involution of $\tilde{C}$. 
Let $f_* \colon JC\to J\tilde{C}$ be the
push-forward map defined by 
$f_*([\sum n_iq_i]) = [\sum n_if(q_i)]$.
Then by the above equality we get
$\Phi_p\circ f_* = -f_*\circ\Phi_p$.
To prove~\eqref{S2} it is enough to show that $f_*(S_2)
= J\tilde{C}(\ff_p)$, which immediately
follows from 
\[
 \Phi_p(f_*(w)) 
 = -f_*(\Phi_p(w))
 = f_*(w),
\]
where $w\in S_2$.

Assume now that $p$ is a supersingular
prime for $E$ such that $p\equiv 23\pmod{24}$.
Hence the Jacobian variety $JC$ is isogenous
to $E_{\sqrt{2}}\times E_{-\sqrt{2}}$, by~\eqref{isog}
and the fact that $2$ is a square in $\ff_p$.
This gives
\[
 \#S_1 =\#JC(\ff_p) = (p+1)^2.
\]
Moreover the trace of the Frobenius morphism
$\Phi_q$, with $q=p^n$, on $H^1_{\textup{\'et}}(JC,\qq_l)$ 
equals the trace of the Frobenius morphism
on the first \'etale cohomology group
of $E_{\sqrt{2}}\times E_{-\sqrt{2}}$.
Since $E_{\pm\sqrt{2}}\mod p$ is supersingular
then its zeta function is:
\[
 Z(E_{\pm\sqrt{2}},t)=\frac{qt^2+1}{(1-t)(1-qt)}.
\]
By means of the isogeny $JC\sim E_{\sqrt{2}}
\times E_{-\sqrt{2}}$ we deduce that the
zeta function of $C$ is
\begin{equation} \label{zec}
 Z(C,t)=\frac{(qt^2+1)^2}{(1-t)(1-qt)}.
\end{equation}
Observe that $\#C(\ff_p) = \#\tilde{C}(\ff_p)$
since $\#C(\ff_p) + \#\tilde{C}(\ff_p) = 2p+2$
and $\#C(\ff_p)=p+1$, by the above calculation
of the zeta function of $C$.
Hence $C$ and $\tilde{C}$ have the same
zeta function, 
since the two curves have the same number of points over $\mathbb{F}_p$, 
and same number of points over $\mathbb{F}_{p^2}$, being isomorphic over this field.  
This gives $\#JC(\ff_p) = \# J\tilde{C}(\ff_p)=(p+1)^2$,
so that $\# Y(\ff_p) = (p+1)^2$ and $\# X(\ff_p)= (p+1)^2+16p$
as claimed.
\end{proof}

\begin{remark}
Theorem~\ref{susie} shows that the supersingular primes for the curve $E$ are congruent to $5$ or $23$ modulo $24$.  For those congruent to $5$ modulo $24$ the zeta functions of $C$ and of $\tilde{C}$ coincide and are equal to 
\[
Z(C,t) = Z(\tilde{C},t) = \frac{(qt^2-1)^2}{(1-t)(1-qt)} ,
\]
analogous to formula~\eqref{zec}.  We verified this formula for the primes $5$, $149$, $173$, $461$, $1229$, $2213$, $2237$ that are the first $7$ supersingular primes of $E$ congruent to $5$ modulo $24$.  As a consequence, the number of points of $X$ modulo a prime $p$ congruent to $5$ modulo $24$ and supersingular for $E$ is $p^2 + 14p + 1$.
\end{remark}

\end{document}